\documentclass [fleqn,preprint]{article}
\usepackage{latexsym,amsmath,amssymb,epsfig,amsthm}
\usepackage{tabularx}

\newtheorem{theorem}{Theorem}
\newtheorem{lemma}{Lemma}
\newtheorem{corollary}{Corollary}
\newcommand\bbR{\mathbb R}
\newcommand\dt{\Delta t}

\theoremstyle{definition}

\newtheorem{remark}{Remark}

\title{Efficient sum-of-exponentials approximations 
for the heat kernel and their applications} 

\author{Shidong Jiang
\thanks{Department of Mathematical Sciences, New Jersey Institute of
Technology, Newark, New Jersey 07102.
{{\em email}: {\sf shidong.jiang@njit.edu}}
Thw work of this author was supported in part by NSF under 
grant CCF-0905395.}
\and
Leslie Greengard 
\thanks{Courant Institute of Mathematical Sciences, New York University,
New York, NY 10012.
{{\em email}: {\sf greengard@courant.nyu.edu}}
The work of this author was supported in part by 
the U.S. Department of Energy under contract DEFGO288ER25053.}
\and
Shaobo Wang 
\thanks{Department of Mathematical Sciences, New Jersey Institute of
Technology, Newark, New Jersey 07102.
{{\em email}: {\sf sw228@njit.edu}}}
}

\begin{document}
\maketitle

\begin{abstract}
In this paper, we show that 
efficient separated sum-of-exponentials approximations can be constructed
for the heat kernel in any dimension. In one space dimension,
the heat kernel admits an approximation
involving a number of terms that is of the order
$O(\log(\frac{T}{\delta}) (\log(\frac{1}{\epsilon})+\log\log(\frac{T}{\delta})))$
for any $x\in\bbR$ and $\delta \leq t \leq T$,
where $\epsilon$ is the desired precision. 
In all higher dimensions,
the corresponding heat kernel
admits an approximation involving only
$O(\log^2(\frac{T}{\delta}))$ terms for fixed accuracy $\epsilon$.
These approximations can be 
used to accelerate integral equation-based methods 
for boundary value problems governed by the heat equation in 
complex geometry. The resulting algorithms are nearly optimal.
For $N_S$ points in the spatial discretization and $N_T$ time steps,
the cost is $O(N_S N_T \log^2 \frac{T}{\delta})$
in terms of both memory and CPU time for fixed accuracy $\epsilon$. 
The algorithms can be parallelized in a straightforward manner. Several
numerical examples are presented to illustrate the accuracy and stability 
of these approximations.
\end{abstract}

\section{Introduction}

Our study of the heat kernel and its approximation is motivated by an
interest in developing efficient methods for the solution of the 
heat equation
\begin{equation}
\label{heateq1}
 U_t = \Delta U, \quad U(x,0) = U_0(x) \, ,
\end{equation}
subject, for example, to Dirichlet boundary conditions 
\begin{equation}
\label{heatbc}
U(x,t) = f(x,t) 
\big|_{\Gamma(t)}
\end{equation}
in a {\em moving} space-time domain
$\Omega_T = \prod_{t=0}^T\Omega(t)$, where $\Gamma(t)$ is 
the boundary of $\Omega(t)$.
More precisely, we are interested in accelerating
methods based on heat potentials \cite{FRIEDMAN,POGOR}, which seek to 
represent $U$ in terms of a single or double layer potential,
\begin{equation}\label{slp}
S(\sigma)(x,t)=\int_{0}^t\int_{\Gamma(t)} 
\frac{\partial}{\partial n_y}G(x-y,t-\tau)
\sigma(y,\tau)ds_yd\tau
\end{equation}
and
\begin{equation}\label{dlp}
D(\sigma)(x,t)=\int_{0}^t\int_{\Gamma(t)} 
\frac{\partial}{\partial n_y}G(x-y,t-\tau)
\sigma(y,\tau)ds_yd\tau \, ,
\end{equation}
respectively, where 
$G(x,t)$ is the heat kernel
\[ G(x,t)=\frac{1}{({4\pi t})^{d/2}}e^{-\frac{|x|^2}{4t}} \, , \]
$n_y$ denotes the unit outward normal to $\Gamma(\tau)$ at the point
$y=y(\tau)$, and $\sigma$ is an unknown surface density.

For the Dirichlet problem (\ref{heateq1}), it is standard to represent
$U$ using the double layer heat potential:
\begin{equation}\label{Urep}
U(x,t) = D(\sigma)(x,t) + V(x,t)
\end{equation}
where $V(x,t)$ denotes the {\em initial potential}
\[ V(x,t) = \int_{\Omega(0)} G(x-y,t) U_0(y)ds_y \, .  \]
Imposing the boundary condition and using standard jump 
relations \cite{FRIEDMAN,POGOR} leads to the 
Volterra integral equation of the second kind
\begin{equation}
\label{volterraeq}
-\frac{1}{2}\sigma(x,t) + D^\ast(\sigma)(x,t) = f(x,t) - V(x,t)
\end{equation}
on the space-time boundary
$\Gamma_T = \prod_{t=0}^T\Gamma(t)$.
Here, $D^\ast(\sigma)(x,t)$ denotes the principal value of the 
double layer potential.

It is convenient both analytically and numerically to decompose the
double layer potential into two pieces: a {\em history} part $D_H$ and a 
{\em local} part $D_L$.
Letting $\delta$ be a small positive parameter, we write 
\[ 
  D(\sigma)(x,t) = D_H(\sigma)(x,t,\delta) + D_L(\sigma)(x,t,\delta) ,
\]
where
\begin{eqnarray}
D_H(\sigma)(x,t,\delta) = 
\int_{0}^{t-\delta}\int_{\Gamma(t)} 
\frac{\partial}{\partial n_y}G(x-y,t-\tau)
\sigma(y,\tau)ds_yd\tau 
\label{eq:dnh}
\end{eqnarray}
and 
\begin{eqnarray}
D_L(\sigma)(x,t,\delta) =  
\int_{t-\delta}^t\int_{\Gamma(t)} 
\frac{\partial}{\partial n_y}G(x-y,t-\tau)
\sigma(y,\tau)ds_yd\tau \, .
\label{eq:dnl}
\end{eqnarray}

A major advantage of the integral equation formulation is that 
high order accuracy can be achieved through the use of 
suitable quadratures, even when the boundary is in motion
(see, for example, \cite{li2}). A second advantage is that the
 equation (\ref{volterraeq}) can be written in the form
\begin{equation}
\label{volterraeq2}
-\frac{1}{2}\sigma(x,t) + D_L^\ast(\sigma)(x,t) = f(x,t) - V(x,t) - 
D_H^\ast(\sigma)(x,t)
\end{equation}
Since the history part is already known and the norm of $D_L$ can be 
shown to be of the order $\| D_L \| = O(\sqrt{\delta})$ \cite{GL:PDE},
the equation (\ref{volterraeq2}) can be solved by $2k$ steps of fixed
point iteration to an accuracy of $\delta^k$. 

There is a substantial engineering literature on boundary integral
eqautions for heat flow (see, for example, \cite{BR:EDITOR,MP:EDITORS}).
In the mathematical literature, it is worth noting the work of 
McIntyre \cite{mcintyre}, Noon \cite{Noon}, Chapko and Kress \cite{chapko},
and the survey article of Costabel \cite{costabel}.

The major disadvantage of 
numerical methods based on heat potentials, however, is their 
history-dependence. Simply evaluating
$D \sigma$ at a sequence of time steps
$t_n = n \Delta t$, for $n=0,\dots,N_T$, clearly
requires an amount of work of the order $O(N_T^2 N_S^2)$,
where $N_S$ denotes the number of points in the discretization of 
the boundary.
While direct discretization methods have been used for
(\ref{volterraeq2}), in the absence of fast algorithms it is 
difficult to argue that integral equation methods would be methods of 
choice for large-scale simulation.
In the last two decades, however, a variety of schemes have been developed
to overcome this obstacle.
The scheme  of \cite{greengard2} used discrete Fourier methods
to represent the history part, while \cite{BRAT:MEIRON} replaced the 
Fourier representation with a regular (physical space) grid on which
to update $D_H$. In \cite{greengard1}, the problem of exterior heat
flow was considered using the continuous Fourier transform in the 
spatial variables.
Sethian and Strain \cite{SETHSTR} and 
Ibanez and Power \cite{ibanez} developed variants of the fast 
algorithm of \cite{greengard2} in the analysis of solidification, melting
and crystal growth.
More recently,
Tausch \cite{tausch} developed an interesting space-time 
``fast-multipole-like" method 
that also overcomes the cost of history-dependence (although it involves 
a hierarchical decomposition of the entire space-time domain).

A somewhat different approach to overcoming history dependence is based
on using the Laplace transform in the time variable, leading to 
what are sometimes called "Laplace transform boundary element methods"
\cite{costabel,hagstrom,lopez3,schadle}.
That is the approach we consider here, for the following reasons:
\begin{enumerate}
\item The Fourier methods of \cite{greengard1,greengard2} assume that
the spatial domain of interest is finite (even when considering exterior
problems). Both the computation of the Fourier modes and the evaluation
of the solution at large distances can involve highly oscillatory integrals.
\item The required number of Fourier modes in \cite{greengard1,greengard2} 
is $O \left(  \left( \frac{a}{\sqrt{\delta}} \right)^d \right)$,
where $a$ is a bound on the extent of the domain in each spatial direction.
This makes the method inefficient for small $\delta$.
\end{enumerate}
Using the Laplace transform avoids both of these difficulties and 
leads to a sum-of-exponentials approximation of the heat kernel that
is asymptotically optimal (although in the end, hybrid schemes may 
yield better constants).

Sum-of-exponentials approximations of convolution kernels 
have many applications in scientific 
computing. They permit, for example, the construction of diagonal forms for 
translation operators. We refer the reader to 
\cite{cheng,yarvin} for their use in the elliptic case in accelerating
fast multipole methods. They also permit the development of 
highly efficient nonreflecting boundary conditions
for the wave and the Schr\"{o}dinger equations
\cite{alpert1,alpert2,hagstrom,jiang1,jiang2,lubich}.

Function approximation using sums of exponentials is
a highly nonlinear problem, so that
the numerical construction of such approximations is nontrivial. 
In a series of papers, Beylkin and Monz\'{o}n 
\cite{beylkin1,beylkin2,beylkin3}
carried out a detailed investigation and developed 
efficient and robust algorithms when given function
values on a fixed interval. 
In some cases, the function of interest can be represented as a parametrized 
integral with exponential functions in the integrand, in which case 
generalized Gaussian quadrature methods \cite{bremer,ma} can also be used. 
In other cases, however, the function being approximated,
say $f(t)$ is only accessible as the 
inverse Laplace transform of an explicitly computable function, 
say $\hat{f}(s)$.
If $\hat{f}(s)$ is sectorial (i.e., holomorphic on the complement of some acute 
sector containing the negative real axis for $s \in \mathbb{C}$), then 
the truncated trapezoidal or midpoint rule can be used
in conjunction with carefully chosen contour integrals, leading to efficient
and accurate sum-of-exponentials approximations. 
L\'{o}pez-Fern\'{a}ndez, Palencia, and Sch\"{a}dle, for example, have 
made effective use of various hyperbolic contours \cite{lopez1,lopez2}. 
On the other hand, if the Laplace transform $\hat{f}(s)$
does not have such well-defined properties, one may instead try to find a 
sum-of-poles approximation in the $s$-domain, from which
a sum-of-exponentials approximation for $f(t)$ can be obtained
by inverting the sum-of-poles approximation analytically 
(see, for example, \cite{xu}). 

In this paper, we construct
efficient {\em separated} sum-of-exponentials 
approximations for the free-space heat kernel. 
In particular, We show that the one-dimensional heat kernel
$\frac{1}{\sqrt{4\pi t}}e^{-|x|^2/(4t)}$ admits an approximation 
of the form
\[ \frac{1}{\sqrt{4\pi t}}e^{-|x|^2/(4t)} \approx
\sum_{i=1}^{N_1} w_i e^{-s_i t}e^{\sqrt{s_i} |x|}  \]
for any $t\in [\delta,T]$ and $x\in \bbR$, 
where $N_1=O(\log(\frac{T}{\delta}) (\log(\frac{1}{\epsilon}) +\log\log(\frac{T}{\delta})))$ 
where $\epsilon$ is the desired precision. In the $d$-dimensional 
case ($d>1$), the heat kernel 
$\frac{1}{(4\pi t)^{d/2}}e^{-|x|^2/(4t)}$ admits 
an approximation of the form 
\begin{equation}
\label{expsumrep}
 \frac{1}{(4\pi t)^{d/2}}e^{-|x|^2/(4t)} \approx
\sum_{j=1}^{N_2} \tilde{w}_j e^{-\lambda_j t} 
\sum_{i=1}^{N_1} w_i e^{-s_i t}e^{\sqrt{s_i} |x|} 
\end{equation}
for all $t\in [\delta,T]$ 
and $x\in \bbR^d$, where 
$N_1=O(\log(\frac{T}{\delta}) (\log(\frac{1}{\epsilon})+\log\log(\frac{T}{\delta})) )$ and
$N_2=O\left( \log\left(\frac{1}{\epsilon}\right)\cdot \left(\log\frac{T}{\delta}+\log\frac{1}{\epsilon}\right)\right)$.
Both our construction and proof draw on earlier work, especially
\cite{beylkin3,lopez2}. 

The paper is organized as follows. In Section 2, we collect some useful results
mainly from \cite{beylkin3,lopez2}.
After developing the sum-of-exponentials representations in 
Section \ref{sec2}, we then discuss 
their application to boundary value problems governed by the heat equation in Section 4.
The basic idea is to use the 
sum-of-exponentials approximation for the computation of the 
history part (such as $D_H(\sigma)$ above), since the parameter
$\delta$ separates the time integration variable $\tau$ from the current
time $t$ so that $t-\tau \in [\delta,T]$. 
Since the temporal dependence for each term in 
(\ref{expsumrep}) involves a simple
exponential, the convolution in time can be easily computed
using standard recurrence relations,
as in \cite{greengard1,greengard2}. 
Furthermore, the convolutions in space can be evaluated
by a variety of fast algorithms, such as variants of 
the fast multipole method. These issues are discussed in 
section \ref{sec:applications}. With $\delta = \Delta t$, where 
$\Delta t$ is a fixed time step,
the total computational 
cost for evaluating $D_H(\sigma)$ is easily seen to be of the order
$O \left( N_S N_T \log(N_S) \log^2(N_T) \right)$,
and the memory
requirements are of the order $O \left( N_S \log^2(N_T) \right)$ for fixed accuracy $\epsilon$.

\section{Analytical preliminaries}
In this section, we collect some results from \cite{beylkin3,lopez2} which will be used 
in subsequent sections.

The following lemma provides an error estimate for the sum-of-exponentials 
approximation
obtained by the truncated trapezoidal rule discretization of 
a certain contour integral in the Laplace domain.
\begin{lemma}\label{lemma1}
\emph{[Adapted from \cite{lopez2}.]}
Suppose that $U(z)$ is holomorphic on $W=\mathbb{C}\setminus 
(-\infty,0]$ and satisfies the estimate 
\begin{equation}\label{1.1}
\|U(z)\| \leq \frac{1}{2|z|^{1/2}}
\end{equation}
for $z\in W$. Suppose that $u(t)=\frac{1}{2\pi i}\int_{\Gamma}e^{tz}U(z)dz$ is the inverse Laplace
transform of $U$. Suppose further that $\alpha$ and $\beta$ satisfy the condition
$0<\alpha-\beta<\alpha+\beta<\frac{\pi}{2}$, $0<\theta<1$, and that 
$\Gamma$ is chosen to be the 
left branch of the hyperbola
\begin{equation}\label{1.2}
\Gamma=\{\lambda T(x)
: x\in \mathbb{R}\},
\end{equation}
where $T(x)=(1-\sin(\alpha+ix))$. Finally, suppose that $u_n(t)$ is the approximation
to $u(t)$ given by the formula
\begin{equation}\label{1.3}
u_n(t)=-\frac{h\lambda}{2\pi i}\sum_{k=-n}^n U(\lambda T(kh))T'(kh)e^{\lambda T(kh)t}.
\end{equation}
Then the choice of parameters
\begin{align}
\label{1.4}
h &=\frac{a(\theta)}{n}, \\
\label{1.5}
\lambda &=\frac{2\pi \beta n (1-\theta)}{T a(\theta)}, \\
\label{1.6}
a(\theta) &=
\operatorname{arccosh}\left(\frac{2T}{\delta (1-\theta)\sin\alpha}\right) 
\end{align}
leads to the uniform estimate on $\delta\leq t \leq T$,
\begin{equation}\label{1.7}
\|u(t)-u_n(t)\|\leq \frac{1}{\sqrt{t}}\phi(\alpha,\beta)\cdot L(\lambda \delta \sin(\alpha-\beta)/2)
\cdot e^{-\frac{2\pi \theta \beta n}{a(\theta)}},
\end{equation}
where 
\begin{equation}\label{1.8}
\phi(\alpha,\beta)=\frac{2}{\pi}\sqrt{\frac{1+\sin(\alpha+\beta)}{1-\sin(\alpha+\beta)}}(e\sin(\alpha-\beta))^{1/2},
\end{equation}

\begin{equation}\label{1.9}
L(x)=1+|\ln(1-e^{-x})|.
\end{equation}
\end{lemma}

\begin{proof}
Choosing $s=\frac{1}{2}$, $\mu=\frac{1}{2}$ in the first equation of 
Remark 2 of \cite{lopez2}, we have the estimate
\begin{eqnarray}
\|u(t)-u_n(t)\| &\leq&  \hspace{1in} \nonumber \\
\label{1.10}
& & \hspace{-1.2in} 
\frac{\phi(\alpha,\beta) e^{\lambda t}}{2\sqrt{t}} \,  L(\lambda t \sin(\alpha-\beta)/2)
 \left(\frac{1}{e^{2\pi \beta/h}-1}+\frac{1}{e^{\lambda t \sin \alpha \cosh(nh)/2}}\right).
\end{eqnarray}
With the choice of parameters given by \eqref{1.4}-\eqref{1.6}, it is easy to see that
\begin{align}\label{1.11}
e^{\lambda t}&=e^{\frac{2\pi \beta n (1-\theta) t}{T a(\theta)}}\leq e^{\frac{2\pi \beta n (1-\theta)}{a(\theta)}}, \qquad \delta\leq t\leq T,\\
\label{1.12}
\frac{1}{e^{2\pi \beta }{h}-1}&\approxeq e^{-2\pi\beta/h} = e^{-\frac{2\pi\beta n}{a(\theta)}},\\
\label{1.13}
\frac{1}{e^{\lambda t \sin \alpha \cosh(nh)}/2}&=e^{-\frac{2\pi \beta n t}{a(\theta)\delta}}
\leq e^{-\frac{2\pi\beta n}{a(\theta)}}, \qquad \delta\leq t\leq T,
\end{align}
Finally, it is easy to see that $L(x)$ is decreasing in $x$ and thus
\begin{equation}\label{1.14}
L(\lambda t \sin(\alpha-\beta)/2)\leq L(\lambda \delta \sin(\alpha-\beta)/2), \qquad \delta\leq t\leq T.
\end{equation}

Substituting \eqref{1.11}-\eqref{1.14} into \eqref{1.10}, we obtain \eqref{1.7}.
\end{proof}

\begin{remark} \label{rem1}
The parameters $\alpha$, $\beta$, and $\theta$ are available for 
optimization. 
For our problem, i.e., the sum-of-exponentials approximation of 
the 1D heat kernel, we choose $\alpha=0.8$, $\beta=0.7$
in \eqref{1.8}. We have also
tested various values of $\theta$ in $(0,1)$. 
Numerical experiments indicate that the number of nodes needed
for a prescribed accuracy is relatively insensitive when $\theta$ is
in the range $[0.85,0.95]$.
\end{remark}

Combining Lemma \ref{lemma1} and Remark \ref{rem1}, we have the following Corollary.
\begin{corollary}\label{cor1}
Suppose that $0<\epsilon<0.1$ is the prescribed relative error and that $T\geq 1000\delta>0$. Then
under the conditions of Lemma \ref{lemma1}, the following 
estimate holds
\begin{equation}\label{1.15}
\|u(t)-u_n(t)\|\leq \frac{1}{\sqrt{t}}\cdot \epsilon, \qquad \delta\leq t\leq T,
\end{equation}
if the number of exponentials $n$ satisfies the following estimate:
\begin{equation}\label{1.20}
n=O\left(\left(\log\left( \frac{1}{\epsilon}\right)
+\log \log \left(\frac{T}{\delta}\right)\right)
\log \left(\frac{T}{\delta}\right)\right).
\end{equation}
\end{corollary}

\begin{proof}
We choose the parameters $\alpha$, $\beta$, and $\theta$ as in Remark \ref{rem1}.
The factor $\phi(\alpha,\beta)$ in \eqref{1.7} is just a fixed constant independent of $n$, 
$T$, and $\delta$. For $T\geq 1000\delta>0$, 
the parameter $a(\theta)$ defined 
in \eqref{1.6} satisfies the estimate 
\begin{equation}\label{1.21}
a(\theta)= O(\log \frac{T}{\delta}).
\end{equation}
Moreover, the function $L$ defined in \eqref{1.9} is decreasing for $x>0$, $L(x)\approx |\ln x|$ as $x\rightarrow 0^+$, 
and $L(x)\rightarrow 1$ as $x\rightarrow +\infty$. 
Combining this observation with the assumption \eqref{1.20}, 
we observe that the factor $L(\lambda \delta \sin(\alpha-\beta))$
in \eqref{1.7} satifies the estimate
\begin{equation}\label{1.22}
L(\lambda \delta \sin(\alpha-\beta)) \leq C \log\left(\frac{T}{\delta}\right).
\end{equation}
Substituting \eqref{1.21} and \eqref{1.22} into \eqref{1.7}, we obtain
\begin{equation}\label{1.23}
\|u(t)-u_n(t)\|\leq \frac{C_1}{\sqrt{t}} \cdot \log\left(\frac{T}{\delta}\right)
\cdot e^{-C_2n/\log\left(\frac{T}{\delta}\right)}.
\end{equation}
It is then easy to see that \eqref{1.15} follows if $n$ satisfies \eqref{1.20}.
\end{proof}

\begin{remark}
For most practical cases, the $\log \log \left(\frac{T}{\delta}\right)$ factor is much smaller
than the $\log\left( \frac{1}{\epsilon}\right)$ factor. Thus for a fixed precision $\epsilon$, 
we will simply say that $n=O\left(\log \left(\frac{T}{\delta}\right)\right)$.
\end{remark}

\begin{remark}
The hyperbolic contour is chosen in such a way that the horizontal 
strip $D_{\beta}=\{z\in \mathbb{C}: |\text{Im} z|\leq \beta\}$ is transformed
into a region bounded by the left branches of two hyperbolas defined
as in \eqref{1.2}, but with $x$ replaced by $x\pm i\beta$. The reason 
that such contour is chosen is the (well-known) fact that the trapezoidal
rule converges exponentially fast for functions holomorphic on a horizontal
strip containing the real axis (see, for example, \cite{stenger1,stenger2}).
We have actually used the midpoint rule to eliminate the occurence of
a node directly on the real axis. 
There is almost no difference in terms of accuracy, but this
allows us to assume that all nodes lie in the upper
half plane in actual computation.
\end{remark}

\begin{remark}
It is likely that other contours would yield similar results.
Trefethen {\em et al.} \cite{trefethen} have analyzed this issue 
with great care and
presented a detailed comparison of various contours 
(hyperbolic, parabolic, and Talbot contours \cite{talbot})
for inverting sectorial Laplace transforms, though they are mainly concerned with 
the efficiency of various contours for a fixed time $t$. 
\end{remark}

We will also need the following lemma, adapted from Theorem 5 in \cite{beylkin3}, which is 
concerned with the efficient sum-of-exponentials approximation for the power function 
$\frac{1}{t^\beta}$.

\begin{lemma}\label{lemma2}
\emph{[Adapted from Theorem 5 in \cite{beylkin3}.]}
For any $\beta\geq \frac{1}{2}$, $1/e\geq \epsilon >0$, and $T>3\delta>0$, 
there exist a positive integer 
\begin{equation}\label{1.17}
N=O\left( \log\left(\frac{1}{\epsilon}\right)\cdot \left(\log\frac{T}{\delta}+\log\frac{1}{\epsilon}\right)\right),
\end{equation}
and positive real numbers $\tilde{w}_i$, $\lambda_i$ ($i=1,\cdots,N$) such that
\begin{equation}\label{1.16}
\left|\frac{1}{t^{\beta}}-\sum_{i=1}^N \tilde{w}_i e^{-\lambda_i t}\right|\leq \frac{1}{t^{\beta}}\cdot \epsilon, \qquad \delta\leq t\leq T.
\end{equation}
Furthermore, for fixed accuracy $\epsilon$, $N=O\left(\log\frac{T}{\delta}\right)$.
\end{lemma}

\begin{remark}
It is pointed out in the caption of Table 1 of \cite{beylkin3} that the 
dependence of the number of terms on accuracy appears to be almost 
linear in $\log\frac{1}{\epsilon}$ rather than
$O\left((\log \frac{1}{\epsilon})^2\right)$, as stated in \eqref{1.17}.
Indeed, Remark 8 in \cite{jiang1} states that the number of exponentials needed to approximate $\frac{1}{t^{1/2}}$ 
for a given relative accuracy $\epsilon$ is $O\left(\log\frac{1}{\epsilon}\left(\log\log\frac{1}{\epsilon}+\log\frac{T}{\delta}\right)\right)$. 
It is straightforward, but tedious, to extend the proof 
of \cite{jiang1} to show that the number of exponentials to achieve a 
relative accuracy of $\epsilon$ for $\frac{1}{t^{n/2}}$ when $n=2,3,4$
is of the same order. This is sufficient for the heat kernel in one, two, and 
three dimensions. 
While this leads to a formal improvement in estimating
the required number of exponentials, the numerical
results from the method of \cite{beylkin3} are actually more efficient than 
those obtained via the explicit construction of \cite{jiang1}. 
Thus, we have left the statement of Lemma \ref{lemma2} as is.
\end{remark}

\begin{remark}
In \cite{braess1}, the error of the approximation of $1/x$ by exponential sums is 
studied in detailed on both finite and infinite intervals.
In \cite{braess2}, Braess and Hackbusch
extend their analysis to the general power function $1/x^\alpha$, $\alpha>0$,
obtaining sharp error estimates for the absolute error.
\end{remark}

\section{Sum-of-exponentials approximation of 
the heat kernel} \label{sec2}

In this section, we first develop 
a separated sum-of-exponentials approximation
for the one-dimensional heat kernel. We then extend the approximation
to arbitrary space dimensions, and to directional derivatives of the heat
kernel, such as the kernel of the double layer potential $D(\sigma)$.

The one-dimensional result is summarized by the following theorem.

\begin{theorem}\label{1dkernel}
Let 
$G(x,t)=\frac{1}{\sqrt{4\pi t}}e^{-\frac{|x|^2}{4t}}$ 
denote the one-dimensional heat kernel and let its 
Laplace transform be denoted by 
$\hat{G}(x,s)=\int_0^{\infty}e^{-st}G(x,t)dt$.
Then, for $s>0$ and $x\in \mathbb{R}$, 
\begin{equation}\label{2.1}
\hat{G}(x,s)=\frac{1}{2\sqrt{s}}e^{-\sqrt{s}|x|}.
\end{equation}
Furthermore, let $0.1>\epsilon>0$ be fixed accuracy and $T\geq 1000\delta>0$ . Then
there exists a sum-of-exponentials approximation
\begin{equation}\label{2.2}
G_A(x,t)=\sum_{k=-n}^n w_k e^{s_k t}e^{-\sqrt{s_k} |x|}
\end{equation}
such that 
\begin{equation}\label{2.3}
|G(x,t)-G_A(x,t)|< \frac{1}{\sqrt{t}}\cdot \epsilon
\end{equation}
for $x\in \mathbb{R}$ and $t\in [\delta,T]$, 
with 
\begin{equation}\label{2.4}
n=\mathcal{O}\left(\log(\frac{T}{\delta}) \left(\log(\frac{1}{\epsilon})
+\log\log\left(\frac{T}{\delta}\right)\right)\right).
\end{equation}
\end{theorem}

\begin{proof}
The formula (\ref{2.1}) is well-known and can be derived from standard
Laplace transform tables.
The natural extension of $\hat{G}(x,s)$ to the complex $s$-plane 
has an obvious branch point at $s=0$, and a branch cut along the negative 
real axis assuming the restriction $-\pi<s\leq \pi$. 
Thus, $\hat{G}(x,s)$ is holomorphic on $W=\mathbb{C}\setminus 
(-\infty,0]$ and satisfies the estimate 
\begin{equation}\label{2.10}
\|\hat{G}(x,s)\| \leq \frac{1}{2|s|^{1/2}}
\end{equation}
for $s\in W$ and all $x\in \mathbb{R}$.
As a result,
$G(x,t)$ can be represented by the inverse Laplace transform:
\begin{equation}\label{2.11}
\begin{aligned}
G(x,t)&=\frac{1}{2\pi i}\int_{\Gamma} e^{st}\hat{G}(x,s)ds\\
&=\frac{1}{4\pi i}\int_{\Gamma} \frac{1}{\sqrt{s}}e^{st-\sqrt{s}|x|}ds,\\
\end{aligned}
\end{equation}
where $\Gamma$ is a simple contour lying in $W=\mathbb{C}\setminus 
(-\infty,0]$, and parametrizable by a regular mapping 
$S:(-\infty,+\infty)\rightarrow \mathbb{C}$ such that
\[ \lim_{x\rightarrow \pm \infty} \text{Im} S(x)=\pm \infty \qquad \text{and} \qquad
\lim_{x\rightarrow \pm \infty} \frac{\text{Re} \left( S(x) \right)}{|x|}<0.\] 
The last condition implies that
\[\text{Re}(z) \leq -b|z|, \qquad \text{as}\ z\rightarrow \infty, z\in \Gamma,\]
for some $b>0$. This forces the integral \eqref{2.11} to 
be absolutely convergent. It is easy to see that the integral is 
independent of the choice of $\Gamma$.

The main result \eqref{2.2} is now a direct
consequence of Lemma 1 and Corollary 1 in Section 2 since 
$\hat{G}$ satisfies the condition \eqref{2.10}. In particular, we have
\begin{equation}\label{2.16}
G_A(x,t)=\sum_{k=-n}^n w_k e^{s_k t}e^{-\sqrt{s_k} |x|},
\end{equation}
where
\begin{align}
\label{2.17}
s_k &=\lambda(1-\sin(\alpha+ikh)), \\
\label{2.18}
w_k &=-\frac{h}{4\pi i \sqrt{s_k}}s'_k,
\end{align}
with $s'_k=-\lambda i\cos(\alpha+ikh)$, and the parameters $h$, $\lambda$ specified in \eqref{1.4}
and \eqref{1.5}, respectively.
\end{proof}

In Table \ref{tab2.1}, we list the number of exponentials needed to approximate
the 1D heat kernel for $x\in \mathbb{R}$ over three different
 time intervals $[\delta,T]$: $I_1=[10^{-3},1]$,
$I_2=[10^{-3},10^3]$, $I_3=[10^{-5},10^4]$, which correspond roughly to $10^3$,
$10^6$, and $10^9$ time steps, respectively. The first column lists the 
maximum error of the approximation computed 
over a $50\times 1000$ 
grid $(x_j,t_k)$ where
$x_0=0$, $x_j=2^{-16+j}$ for $j=1,\cdots,49$, and the $t_k$ 
are $1000$ samples on $[\delta,T]$ chosen to be equispaced
on a logarithmic scale. 
The node locations are plotted in Fig. \ref{fig2.1}.

\begin{table}[ht]
\begin{center}
{\bf \caption{Number of exponentials needed to approximate the 1D heat kernel
for $x\in \mathbb{R}$ over different time intervals: 
$I_1=[10^{-3},1]$,
$I_2=[10^{-3},10^3]$, $I_3=[10^{-5},10^4]$. \label{tab2.1}
}}
\begin{tabular*}{1.0\textwidth}{@{\extracolsep{\fill}}  l  c  c  r  }
\hline
$\epsilon$ & $I_1$ & $I_2$ & $I_3$ \\
\hline
$10^{-3}$ & 15 & 23 & 32\\
$10^{-6}$ & 31 & 50 & 68\\
$10^{-9}$ & 47 & 77 & 105\\
\hline
\end{tabular*}
\end{center}
\end{table}

\begin{figure}[!ht]
\centerline{\includegraphics[trim=20 240 10 240, clip, height=2.5in]
{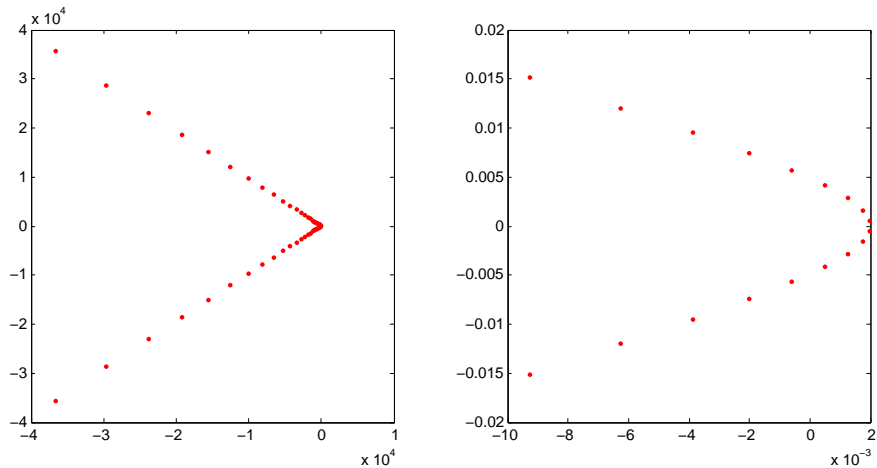}} 
\caption{\emph{The location of the exponential nodes $s_k$ ($k=\pm 1,\cdots,\pm n$ 
for $n=77$) used
in \eqref{2.16} to approximate the 1D heat kernel 
for $t\in [10^{-3},10^3]$ with $9$-digit accuracy. The left figure shows
all nodes, while the right figure is a close-up of the nodes near the origin. All nodes
lie on the left branch of the hyperbola specified in \eqref{1.2}.} 
} \label{fig2.1}
\end{figure}

\begin{remark}
To obtain Table 1, we set $\alpha=0.8$, $\beta=0.7$, $\theta=0.9$ for $I_1$ and $\theta=0.95$ 
for $I_2$ and $I_3$. 
\end{remark}

\begin{remark}
The coefficients $w_i$ are not positive (but
complex) and thus stability is an issue. It is difficult to bound the expression
\[\frac{\sum|w_i \exp(s_i t)|}{|\sum w_i \exp(s_i t)|}\]
analytically for all time. However, we have checked the value of
the expression numerically for all cases presented in Table 1, and
found that it is roughly $1.08$. This suggests that 
the sum-of-exponentials approximation is, indeed, well-conditioned. 
\end{remark}

\begin{remark}
If we make a further change of variable $z=\sqrt{s}$ in \eqref{2.11}, we obtain
\begin{equation}\label{2.19}
G(x,t)=\frac{1}{2\pi i}\int_{\Gamma'} e^{z^2t-z|x|}dz,\\
\end{equation}
where $\Gamma'$ is any contour lying in the sector $\{\frac{\pi}{4}< |\arg z|
\leq \frac{\pi}{2}\}$ of the complex plane. In particular, if $\Gamma'$ is chosen
to be the imaginary axis, then we essentially recover the Fourier integral representation
of the heat kernel (see, for example, \cite{greengard2}).
\end{remark}

Finally, it is worth repeating that
the number of terms required in the sum-of-exponentials approximation
does not depend on the spatial extent of the problem. 

\subsection{Heat kernels in higher dimensions}

Suppose now that we are interested in heat flow in $\bbR^d$, for $d\geq 2$,
where the heat kernel is 
\[ G_d(x,t)=\frac{1}{(4\pi t)^{d/2}}e^{-\frac{|x|^2}{4t}}. \] 
The following theorem describes
an efficient sum-of-exponentials representation.

\begin{theorem}\label{ndkernel}
For any $0.1>\epsilon >0$ and $T\geq 1000\delta>0$,
the heat kernel $G_d(x,t)$ admits the 
following approximation:
\begin{equation}\label{3.1}
\tilde{G}_d(x,t)=\sum_{j=1}^{N_2} \tilde{w}_j e^{-\lambda_j t}
\sum_{k=-N_1}^{N_1} w_k e^{s_k t}e^{-\sqrt{s_k} |x|}
\end{equation}
such that 
\begin{equation}\label{3.2}
|G_d(x,t)-\tilde{G}_d(x,t)|<\frac{1}{t^{d/2}}\cdot \epsilon
\end{equation}
for any $x\in \mathbb{R}^d$, $t\in [\delta,T]$ 
with $N_1$ specified in \eqref{2.4} and $N_2$ specified in
\eqref{1.17}.
\end{theorem}

\begin{proof}
We first rewrite
 the $d$ dimensional heat kernel as a product of two functions:
\begin{equation}\label{3.3}
G_d(x,t)=G_1(x,t)\cdot F(t),
\end{equation}
where
\begin{equation}\label{3.4}
G_1(x,t)=\frac{1}{(4\pi t)^{1/2}}e^{-\frac{|x|^2}{4t}}, \qquad
F(t)=\frac{1}{(4\pi t)^{(d-1)/2}}.
\end{equation}
Similarly, we rewrite $\tilde{G}_d$ as follows:
\begin{equation}\label{3.5}
\tilde{G}_d(x,t)=S_1(x,t)\cdot S_2(t),
\end{equation}
where
\begin{equation}\label{3.6}
S_1(x,t)=\sum_{k=-N_1}^{N_1} w_k e^{s_k t}e^{-\sqrt{s_k} |x|}, \qquad
S_2(t)=\sum_{j=1}^{N_2} \tilde{w}_j e^{-\lambda_j t}.
\end{equation}
Using the triangle inequality, we have
\begin{equation}\label{3.8}
\begin{aligned}
|G_d(x,t)-\tilde{G}_d(x,t)|&\leq |G_1\cdot F-G_1\cdot S_2|+|G_1\cdot S_2-S_1\cdot S_2|\\
&\leq |G_1|\cdot |F-S_2|+(1+\epsilon)|F|\cdot |G_1-S_1|\\
&\leq \frac{3}{ t^{d/2}}\cdot \epsilon,
\end{aligned}
\end{equation}
and the result follows.
\end{proof}

\subsection{The double layer heat potential}

Since we often rely on the double layer potential in 
integral equation methods, it is worth writing down a sum-of-exponentials
approximation for this case as well.
We denote the kernel of the double layer heat potential by
\[ D(x,y; t)=\frac{\partial G(x-y,t)}{\partial n_y}=
\frac{(x-y)\cdot n_y}{(4\pi t)^{d/2} 2t}e^{-\frac{|x-y|^2}{4t}}.
\] 

\begin{theorem}\label{doublekernel}
Let $0.1>\epsilon > 0$ be fixed accuracy, $R>1$ and $T\geq 1000\delta>0$.
Then there exists a sum-of-exponentials approximation
\begin{equation}\label{4.1}
D_A(x,y;t)=\sum_{j=1}^{N_2} \tilde{w}_j e^{-\lambda_j t}\sum_{k=-N_1}^{N_1} w_k e^{s_k t}e^{-\sqrt{s_k} |x-y|}(x-y)\cdot n_y
\end{equation}
such that 
\begin{equation}\label{4.2}
|D(x,y;t)-D_A(x,y;t)|<\frac{1}{t^{(d+2)/2}} \cdot \epsilon
\end{equation}
for $|x-y|\leq R$ and $t\in [\delta,T]$.
Here $N_1$ and $N_2$ are as follows:
\begin{equation}\label{4.3}
N_1=\mathcal{O}\left(\log(\frac{T}{\delta})\left(\log(\frac{1}{\epsilon})+
\log\log\left(\frac{T}{\delta}\right)+
\log R\right)\right)
\end{equation}
and 
\begin{equation}\label{4.4}
N_2=O\left( (\log\left(\frac{1}{\epsilon}\right)+\log R)\cdot
  \left(\log\frac{T}{\delta}+\log\frac{1}{\epsilon}+\log
    R\right)\right).
\end{equation}

\end{theorem}

\begin{proof}
We first introduce $\hat{\epsilon}=\epsilon/R$. Then by Theorem 1,
there exist
$N_1=\mathcal{O}\left(\log(\frac{T}{\delta})\log(\frac{1}{\hat{\epsilon}})\right)
=\mathcal{O}\left(\log(\frac{T}{\delta})\left(\log(\frac{1}{\epsilon})+
\log\log\left(\frac{T}{\delta}\right)+
\log
  R\right)\right)$,
coefficients $w_k$, and nodes $s_k$ ($k=1,\cdots,N_1$) such that
\begin{equation}\label{4.5}
|G_1(x,t)-\sum_{k=-N_1}^{N_1} w_k e^{s_k t}e^{-\sqrt{s_k} |x|}
|<\frac{1}{\sqrt{t}}\cdot \hat{\epsilon},
\end{equation}
for $x\in \mathbb{R}$ and $t\in [\delta,T]$.
 
Changing $x$ to $x-y$ and multiplying both sides of \eqref{4.5} by
$(x-y)\cdot n_y$, we
obtain
\begin{equation}\label{4.6}
\begin{aligned}
&\left|\frac{(x-y)\cdot n_y
}{(4\pi t)^{1/2}
  }e^{-\frac{|x-y|^2}{4t}}-\sum_{k=-N_1}^{N_1} w_k e^{s_k
    t}e^{-\sqrt{s_k} |x-y|}\cdot ((x-y)\cdot n_y)\right|\\
&\leq \frac{|(x-y)\cdot n_y|}{(4\pi t)^{1/2}} \cdot \hat{\epsilon}\\
&\leq \frac{R}{(4\pi t)^{1/2}} \cdot \frac{\epsilon}{R}\\
&\leq \frac{1}{(4\pi t)^{1/2}} \cdot \epsilon\\
\end{aligned}
\end{equation}
for all $|x-y|\leq R$ and $t\in [\delta,T]$.

Similarly, by Lemma 2, we have
\begin{equation}\label{4.7}
\left|\frac{1}{t^{(d+1)/2}}-\sum_{i=1}^{N_2} \tilde{w}_i e^{-\lambda_i t}\right|\leq \frac{1}{t^{(d+1)/2}}\cdot \epsilon/R, \qquad \delta\leq t\leq T,
\end{equation}
where $N_2$ is given by \eqref{4.4}.

The result is then obtained by an argument almost identical to that in the proof
of Theorem \ref{ndkernel}. 
\end{proof}

\begin{remark}
As discussed in Remark 5, more involved analysis could replace the
order estimate 
$O((\log\frac{1}{\epsilon}+\log R)^2)$ 
for $N_2$ in \eqref{4.4} with $O(\log\frac{1}{\epsilon}+\log R)$.
\end{remark}

\begin{remark}
It is worth noting that Theorems 1--3 provide what are, in essence,
relative error estimates. Our numerical experiments also indicate that the
$\log R$ dependence in $N_1$ and $N_2$ and the restriction on 
$x-y$ are somewhat artificial since
$D(x,y;t)$ and $D_A(x,y;t)$ are exponentially small for large $(x-y)$.
\end{remark}

\section{Applications}  \label{sec:applications}

We return now to a consideration of the exterior Dirichlet problem
governed by the heat equation in eqs.
\eqref{heateq1} and \eqref{heatbc}. 
We let $\delta=(k-1)\Delta t$ in order to obtain
a $k$th order scheme and consider the calculation of $D_L$ and $D_H$
separately. To simplify the notation, we fix the relative accuracy
$\epsilon$, the computational domain size $R$, and
suppress the dependence of the complexity on $\epsilon$ and $R$ in the following discussion.
\subsection{Evaluation of the local part $D_L$} \label{sec:loc}

The discretization of the local part of the double layer potential $D_L$ 
is discussed in \cite{li2} for both stationary and moving boundaries. 
We limit our attention here to 
the case of a stationary boundary, $\Gamma(t) = \Gamma(0)$, for the sake
of simplicity.  The basic idea is to expand the density
$\sigma(y,t)$ on $[t-\delta,t]$ for each $y$ in the form
\[ \sigma(y,\tau) = \sigma_0(y) + (t-\tau) \sigma_1(y) + \dots +
\frac{ (t-\tau)^{k-1}}{(k-1)!} \sigma_{k-1}(y) + O((t-\tau)^k).
\]
The functions $\sigma_0(y),\dots, \sigma_{k-1}(y)$ are obtained
from the function values  $\sigma(y,t-j\Delta t)$ for $j=0,\cdots, k-1$.

Substituting the above expression into \eqref{eq:dnl} and changing the order of integration in time and space, 
we obtain
\begin{align}
\label{expintseriesdlp}
D_L(\sigma)(x,t,\delta) =  \bigg[
&\int_\Gamma D_{L,0}(x,y) \sigma_0(y) \, ds_{y} 
+ 
\int_\Gamma D_{L,1}(x,y) \sigma_1(y) \, ds_{y} 
 + \dots \\
 &+ 
\frac{1}{(k-1)!}
\int_\Gamma D_{L,k-1}(x,y) \sigma_{k-1}(y) \, ds_{y} \bigg]
 + O( (t-\tau)^{k+1/2} . \nonumber
 \end{align}
where 
$D_{L,k}(x,y)$ is given by 
\begin{equation}
\label{time_kernels}
D_{L,k}(x,y) = 
\int_{t-\delta}^t 
\frac{\partial}{\partial n_y}G_d(x-y,t-\tau)
(t-\tau)^{k-2} \, d\tau . 
\end{equation}
For an analgous treatment of unsteady Stokes potentials,
see \cite{jiang3}.

\begin{remark}
The above procedure provides a robust, high-order
marching scheme that is insensitive to the 
complexity of the geometry. Simpler time-marching schemes are
subject to {\em geometrically induced stiffness}, discussed at some length
in \cite{jiang3,li2}.
\end{remark}

The kernels $D_{L,k}$ can be computed in closed form in terms 
of the exponential integral function ${\rm Ei}(1,x)$ in two dimensions
and the error function ${\rm erf}(x)$ in three dimensions.
Each of these kernels is
singular (or weakly singular), but smooth away from the diagonal ($x=y$).
The discretization of the spatial integrals, therefore, requires
techniques much like those needed for elliptic layer potentials.
We refer the reader to \cite{alpert0} and \cite{qbx} for a discussion
of high order accurate rules. 

When computing
$D_{L,k}$ at each boundary point, we would also like to avoid the 
$O(N_S^2)$ work that would be required by naive evaluation of the integral.
A large number of fast algorithms are now available to reduce the cost
of this step to  $O(N_S)$ or $O(N_S \log N_S)$. 
These include fast multipole methods, kernel-independent fast multipole
methods \cite{kifmm}, HSS and {\cal H}-matrix methods
\cite{HSS,Hmat}, and HBS or recursive 
skeletonization methods \cite{ho,gillman}. 
We propose to use the approach developed in \cite{ho}. That is, each of 
the operators $D_{L,j}(x_m,y_n)$
($j=1,\cdots, k-1$) will be compressed once, with subsequent applications
of the operator computed in  
near optimal complexity time with smaller constant prefactors than
analytically-based fast multipole methods. It is also possible, of course,
to store a compressed version of the full operator $D_L$ or 
$(-\frac{1}{2}I+D_L^\ast)^{-1}$.
While the storage costs are large, they scale nearly linearly with $N_S$.
This is ideal for long-time simulation in stationary domains,
where the same linear system is solved at each time step with a different 
right-hand side.

\subsection{Evaluation of the history part $D_H$}
For the history part, approximating the kernel by $D_A(x,y;t-\tau)$ and
substituting \eqref{4.1} into \eqref{eq:dnh}, we obtain

\begin{equation}\label{6.1}
\begin{aligned}
D_H(\sigma)(x,t)&\approx\int_0^{t-\delta}\int_{\Gamma} 
\sum_{j=1}^{N_2} \tilde{w}_j e^{-\lambda_j (t-\tau)}\\
&\sum_{k=-N_1}^{N_1} w_k e^{s_k (t-\tau)}e^{-\sqrt{s_k} |x-y|}[(x-y)\cdot n_y]
\sigma(y,\tau)ds_yd\tau\\
&=\sum_{j=1}^{N_2} \tilde{w}_j \sum_{k=-N_1}^{N_1} w_k H_{j,k}(x,t),
\end{aligned}
\end{equation}
where each history mode $H_{j,k}$ is given by the formula
\begin{equation}\label{6.2}
H_{j,k}(x,t)=\int_0^{t-\delta}
e^{(-\lambda_j+s_k) (t-\tau)} V_k(x,\tau)
d\tau,
\end{equation}
with $V_k$ given by the formula
\begin{equation}\label{6.3}
V_k(x,\tau)=\int_{\Gamma} e^{-\sqrt{s_k} |x-y|}[(x-y)\cdot n_y]
\sigma(y,\tau)ds_y.
\end{equation}
Here, we have interchanged the order of summation and integration. 

For each fixed $\tau$, $V_k(x,\tau)$ can be discretized using the 
trapezoidal rule and the resulting
summation can again be computed using the fast algorithms in \cite{ho} 
for all $x\in \Gamma$.
The computational cost for this step is 
$\mathcal{O}(N_S \log N_S)$ for each $k$.
Once the $V_k$ have been evaluated, 
each history mode $H_{j,k}$ can be computed recursively, as in 
\cite{greengard2,jiang3}:
\[
H_{j,k}(x,t+\Delta t)= 
e^{(-\lambda_j+s_k) \Delta t} H_{j,k}(x,t) +
\int_{t-\delta}^{t+\Delta t-\delta}
e^{(-\lambda_j+s_k) (t+\Delta t-\tau)} V_k(x,\tau)
d\tau.
\]
(Equivalently, each
history mode $H_{j,k}$ can be seen to satisfy a simple
linear ODE.)
The point is that each history mode can be computed in $\mathcal{O}(1)$
operations at each time step for each $x$.
Since both $N_1$ and $N_2$ are 
$\mathcal{O}(\log(T/\delta))=\mathcal{O}(N_T)$ with $N_T$ 
the total number of time steps, the net computational cost
for the evaluation of the history part at each time step is 
$\mathcal{O}(N_S \log N_S \log N_T +N_S \log^2 N_T)$, with storage requirements
of the order $\mathcal{O}(N_S \log^2 N_T)$.
The computational cost for the entire simulation is 
\[ 
\mathcal{O}(N_S N_T \log N_S \log N_T +N_S N_T \log^2 N_T) \, .
\]

The algorithm of the present paper is embarrassingly parallel in that
the computation of each history mode is independent.
Furthermore, the hierarchical fast algorithms used for
each  $V_k(x,\tau)$ are
themselves amenable to parallelization, and there
is already a substantial body of research and software devoted to that task for 
large-scale problems in three dimensions.

%
%
%
\section{Numerical examples} \label{sec:results}

In this section, we illustrate the accuracy and stability of the sum-of-exponential
approximations for the heat kernels. More precisely, we have implemented
the algorithm outlined in the previous section to solve the exterior Dirichlet problem
governed by the heat equation in two dimensions. The spatial integral in the local parts is 
discretized using the 16th order hybrid Gauss-trapezoidal rule of \cite{alpert0}, while the
spatial integral in the history part is discretized using the trapezoidal rule. 
We have implemented 2nd, 3rd, and 4th order schemes in time (defined by the number
of terms taken in the local Taylor expansion of $\sigma$ in section \ref{sec:loc}),
with numerical experiments carried out using the fourth order version.
We consider four simple boundary curves:
a circle, an ellipse with aspect ratio 2:1,
a crescent, and a smooth hexagram, shown in Figures 1--4, 
respectively. The final time is set to $T=1$ and the
boundary curves are roughly of size $R=5$. We generate boundary data by placing 
heat sources {\em inside} the boundary curve and test the accuracy of our numerical solution
by comparing it with the analytical solution at 20 target points outside the boundary curves.
We use 350 spatial discretization points for the circle, 256 points for the ellipse, 
512 points for the crescent, and 350 points for the hexagram. Since our spatial
integration rules are very high order and the geometry is well-resolved, the accuracy is 
dominated by the discretization error in time.

In each of the tables below, 
$\dt$ is the 
time step, $N_T$ is the total number of time steps, 
$K$ is the condition number of the linear system that
needs to be solved at each time step, 
$E$ is the relative $L^2$ 
error of the numerical solution at
the final time, and $r$ is the ratio of relative $L^2$ errors
for successive time step refinements:
$r(j)=E(j-1)/E(j)$. 

\begin{remark}
Since our examples are intended to demonstrate the accuracy and stability
of the sum-of-exponential approximations for the heat kernels, we have limited 
our attention to modest size problems.
With larger-scale boundary discretizations,
fast solvers will be used to invert the linear systems at each time step 
and to speed up the computation of the spatial integrals in the history part. 
We are currently building such codes, including parallelization.
\end{remark}

\begin{figure}[htbp!]
\begin{center}
  \begin{minipage}[c]{0.33\textwidth}
            \centering\includegraphics[width=.5\textwidth]{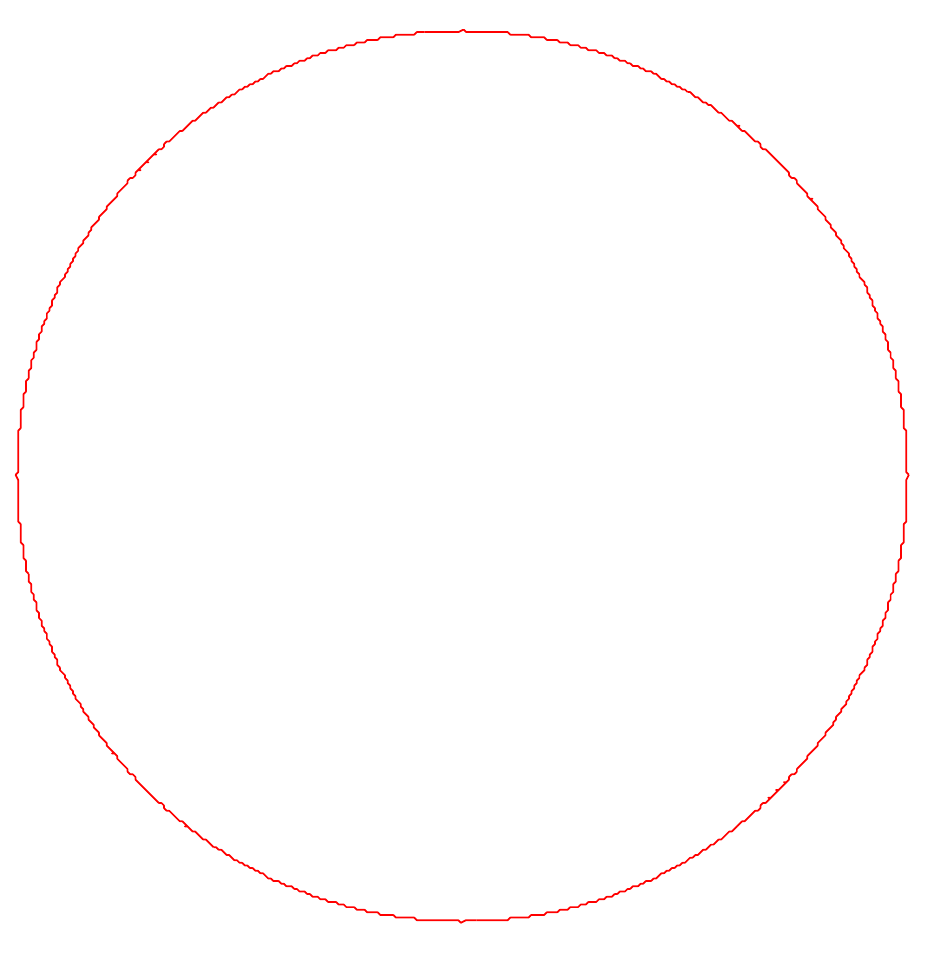}
   \end{minipage}
   \hspace{0.02in}
   \begin{minipage}[c]{0.6\textwidth}
   \begin{tabularx}{\textwidth}{@{\extracolsep{\fill}} |c|c|c|c|c|  }
        \hline
$\dt$ & $N_T$ & $K$  & $E$ & $r$ \\
\hline 
1.00e-1 & 10 & 1.02 & 4.33e-5 &  \\
5.00e-2 & 20 & 1.01 & 3.00e-6 & 14.4 \\
2.50e-2 & 40 & 1.01 & 2.01e-7 & 14.9 \\
1.25e-2 & 80 & 1.01 & 7.57e-9 & 26.5 \\
6.25e-3 & 160& 1.00 & 3.81e-10 & 19.9 \\
\hline
   \end{tabularx}
   \end{minipage}
   \caption{Numerical results for the circle. 
$N_T$ is the total number of time steps, 
$K$ is the condition number of the linear system to be
solved at each time step, $E$ is the relative error at the final time
in $L^2$, and $r$ is the ratio of relative $L^2$ errors
for successive time step refinements:
$r(j)=E(j-1)/E(j)$.}
\label{fig1} 
\end{center}
\end{figure}

\begin{figure}[htbp!]
\begin{center}
  \begin{minipage}[c]{0.33\textwidth}
            \centering\includegraphics[width=.65\textwidth]{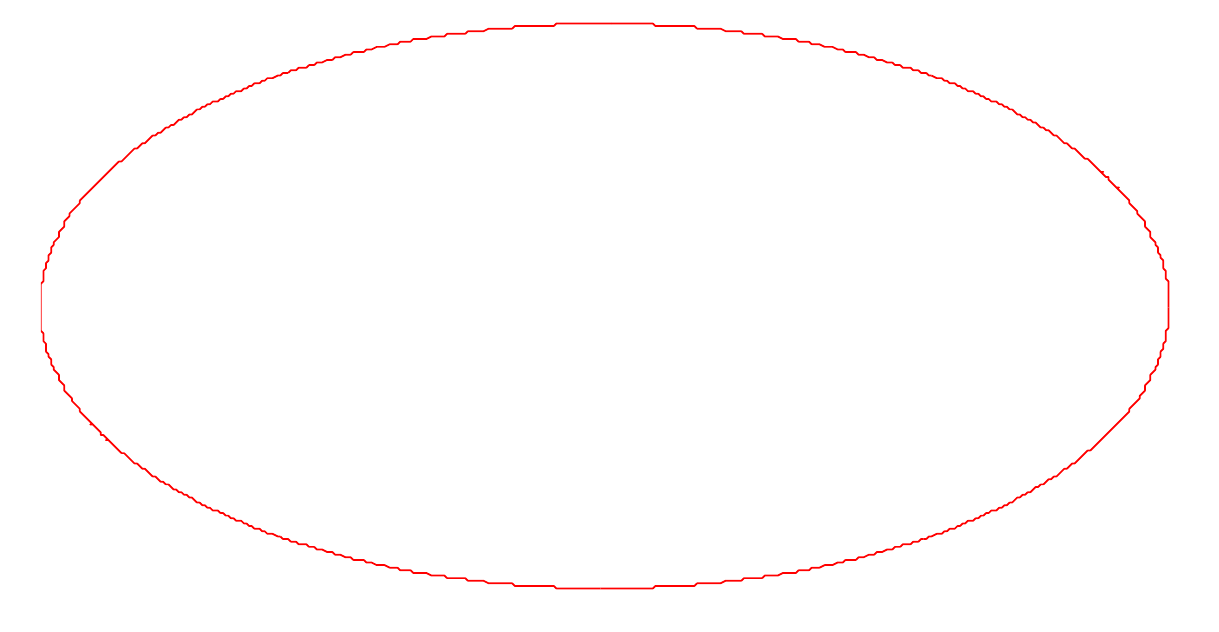}
   \end{minipage}
   \hspace{0.02in}
   \begin{minipage}[c]{0.55\textwidth}
   \begin{tabularx}{\textwidth}{@{\extracolsep{\fill}} |c|c|c|c|c|  }
        \hline 
$\dt$ & $N_T$ & $K$  & $E$ & $r$ \\
\hline 
1.00e-1 & 10 & 1.07 & 1.85e-4 &  \\
5.00e-2 & 20 & 1.05 & 1.55e-5 & 11.9 \\
2.50e-2 & 40 & 1.04 & 1.11e-6 & 13.9 \\
1.25e-2 & 80 & 1.03 & 1.84e-8 & 60.3 \\
6.25e-3 & 160& 1.02 & 6.76e-10 & 27.3 \\
\hline
   \end{tabularx}
   \end{minipage}
   \caption{Numerical results for a 2:1 ellipse.
See text or Fig. \ref{fig1} for explanation.}
\label{fig2} 
\end{center}
\end{figure}

\begin{figure}[htbp!]
\begin{center}
  \begin{minipage}[c]{0.33\textwidth}
            \centering\includegraphics[width=.6\textwidth]{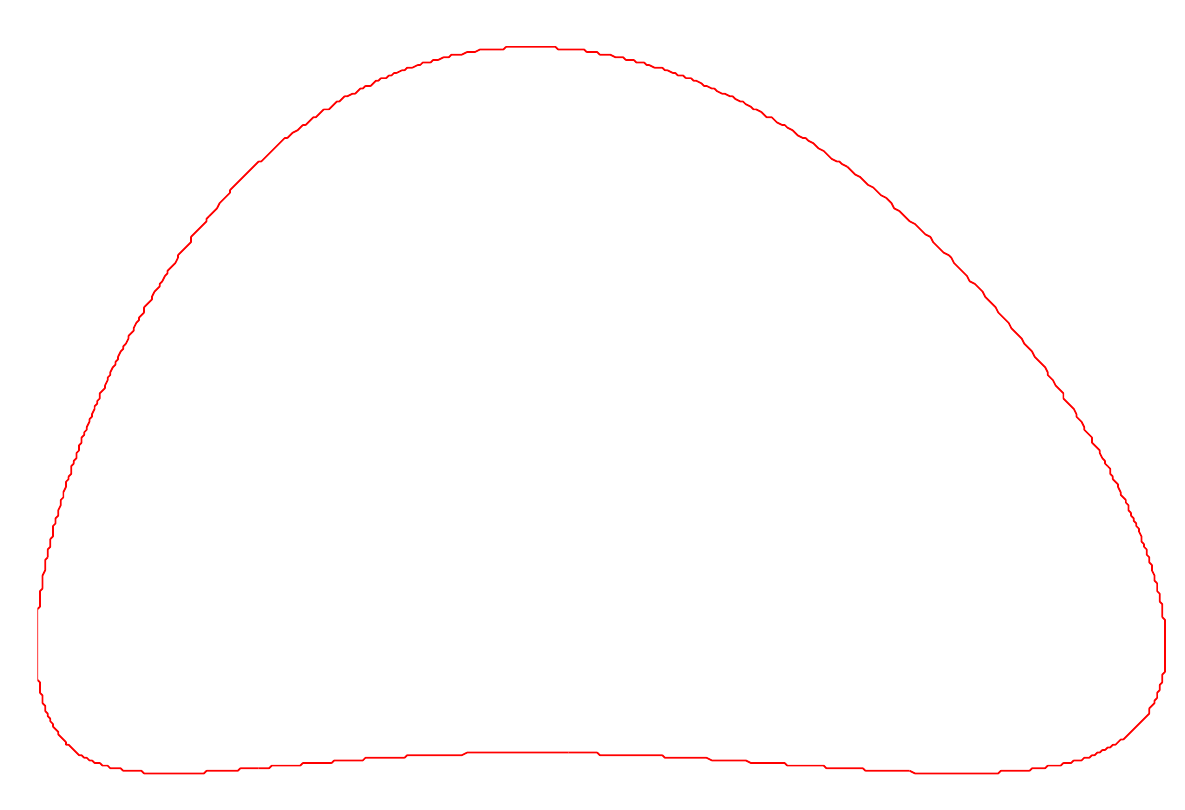}
   \end{minipage}
   \hspace{0.02in}
   \begin{minipage}[c]{0.55\textwidth}
   \begin{tabularx}{\textwidth}{@{\extracolsep{\fill}} |c|c|c|c|c|  }
        \hline
$\dt$ & $N_T$ & $K$  & $E$ & $r$ \\
\hline 
1.00e-1 & 10 & 1.15 & 8.76e-5 &  \\
5.00e-2 & 20 & 1.11 & 3.13e-6 & 28.0 \\
2.50e-2 & 40 & 1.08 & 1.89e-7 & 16.5 \\
1.25e-2 & 80 & 1.06 & 3.16e-9 & 59.8 \\
6.25e-3 & 160& 1.04 & 1.28e-10 & 24.6 \\
\hline
   \end{tabularx}
   \end{minipage}
   \caption{Numerical results for a smooth crescent.
See text or Fig. \ref{fig1} for explanation.}
\label{fig3} 
\end{center}
\end{figure}

\begin{figure}[htbp!]
\begin{center}
  \begin{minipage}[c]{0.33\textwidth}
            \centering\includegraphics[width=.55\textwidth]{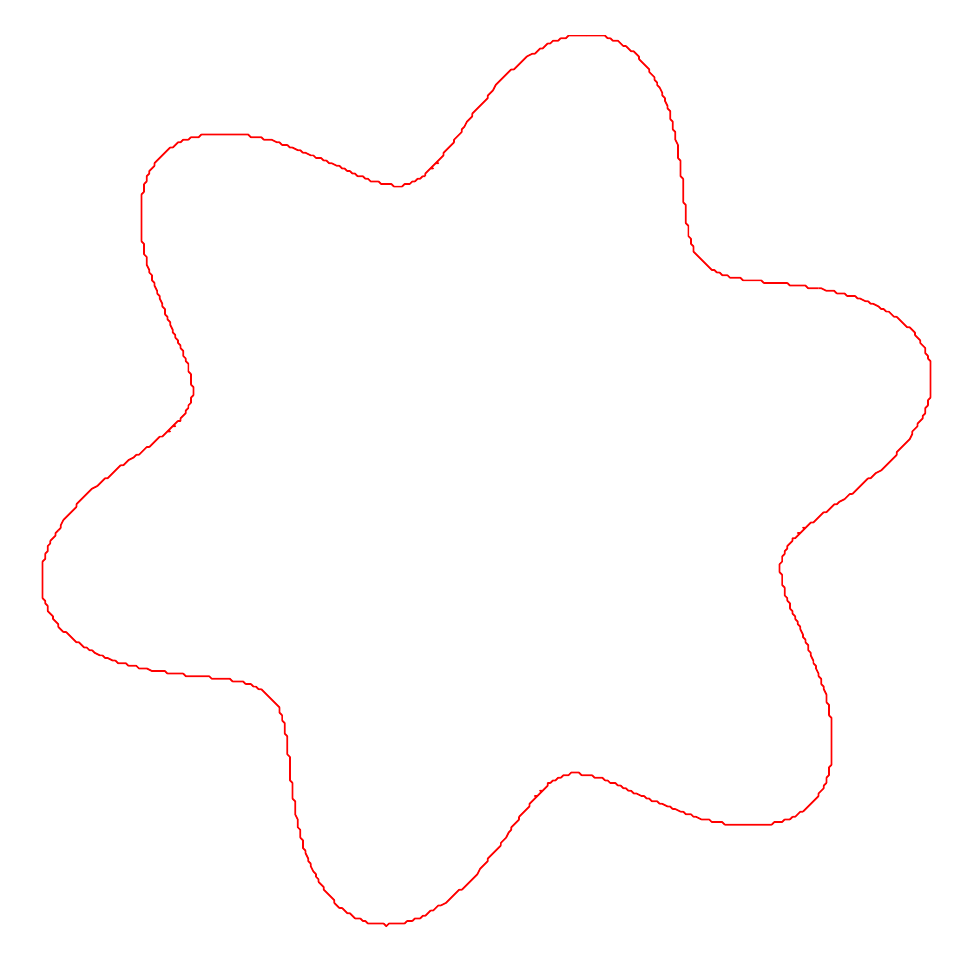}
   \end{minipage}
   \hspace{0.02in}
   \vspace{0.12in}
   \begin{minipage}[c]{0.55\textwidth}
   \begin{tabularx}{\textwidth}{@{\extracolsep{\fill}} |c|c|c|c|c|  }
        \hline
$\dt$ & $N_T$ & $K$  & $E$ & $r$ \\
\hline 
1.00e-1 & 10 & 1.03 & 6.05e-5 &  \\
5.00e-2 & 20 & 1.02 & 3.06e-6 & 19.8 \\
2.50e-2 & 40 & 1.02 & 1.94e-7 & 15.7 \\
1.25e-2 & 80 & 1.01 & 7.80e-9 & 24.5 \\
6.25e-3 & 160& 1.01 & 2.80e-10 & 27.8 \\
\hline
   \end{tabularx}
   \end{minipage}
   \caption{Numerical results for a smooth hexagram.
See text or Fig. \ref{fig1} for explanation.}
\label{fig4} 
\end{center}
\end{figure}

Several observations concerning our results are in order. First, note that
there is no stability issue for large time steps (as expected).
Second, the linear systems which need to be solved at each time step
are extremely well-conditioned, since the compact
operator in the Volterra integral equation has small norm.
Third, the convergence rates are roughly consistent with 
fourth order accuracy (slightly better because of the smoothing behavior of the
heat equation).

Finally, in order for the reader to be able to easily implement the scheme, 
we list the sum-of-exponential approximations for 
the power function
$\frac{1}{t^{3/2}}$ and the 1D heat kernel in Tables \ref{tab4} and \ref{tab5}, respectively.

\begin{table}[htbp!]
\caption{Sum-of-exponential approximation for the power function $\frac{1}{t^{3/2}} \approx \sum_{k=1}^{N}w_ke^{s_k t}$. Here $N=22$. The relative error is $10^{-9}$ for $t\in [10^{-3},1]$. The first column lists the values of $s_k$ and the second column lists the values of $w_k$.}
\begin{center}
\scalebox{0.75}{
\begin{tabular}{|c|c|} \hline
$s_k$  & $w_k$ \\ \hline
 -7.2906159549928551D-001 & 1.4185610815528382D+000 \\
 -3.0048118783719833D+000 & 6.1168393596966855D+000 \\
 -7.1264765899586058D+000 & 1.5711472927135111D+001 \\
 -1.3711947176886415D+001 & 3.4005632951162106D+001 \\
 -2.3888158541731844D+001 & 6.9073954353680932D+001 \\
 -3.9552395196025927D+001 & 1.3693587825572760D+002 \\
 -6.3705812577837150D+001 & 2.6803903557553036D+002 \\
 -1.0092380540293159D+002 & 5.1918775664556608D+002 \\
 -1.5806801843263241D+002 & 9.9539106627427793D+002 \\
 -2.4535153453758187D+002 & 1.8891680379043360D+003 \\
 -3.7789560121839855D+002 & 3.5504717885763807D+003 \\
 -5.7797736857567725D+002 & 6.6101535364473611D+003 \\
 -8.7825970527795960D+002 & 1.2196709015044245D+004 \\
 -1.3264302774624159D+003 & 2.2315687874782194D+004 \\
 -1.9919031777954062D+003 & 4.0515943857625483D+004 \\
 -2.9756729573498451D+003 & 7.3077757126301352D+004 \\
 -4.4253961362173968D+003 & 1.3121449467706907D+005 \\
 -6.5602925574528726D+003 & 2.3547202449444451D+005 \\
 -9.7174072351500563D+003 & 4.2563843229579012D+005 \\
 -1.4450262425240786D+004 & 7.8639206472664094D+005 \\
 -2.1750443538417538D+004 & 1.5127774335123657D+006 \\
 -3.3020803685636522D+004 & 2.6486986425576657D+006 \\
\hline
\end{tabular}
}
\end{center}
\label{tab4}
\end{table}

\begin{table}[htbp!]
\caption{Sum-of-exponential approximation for the 1D heat kernel $G_1=e^{-x^2/4t}/\sqrt{4\pi t}$. $G_1 \approx \sum_{k=-N_1}^{N_1}w_ke^{s_k t}e^{-\sqrt{s_k}x}$. Here $N_1=47$ and only the exponentials in the lower half of the complex plane are listed. The relative error is $10^{-9}$ for $t\in [10^{-3},1]$. The first column lists the values of $s_k$ and the second column lists the values of $w_k$.}
\begin{center}
\scalebox{0.65}{
\begin{tabular}{|c|c|} \hline
$s_k$  & $w_k$ \\ \hline
 +5.2543538566883130D-001 -1.5355921434143971D-001i & 3.3406450069243053D-002 + 7.8694879334856565D-004i\\
 +4.5124072851051300D-001 -4.6901795847816991D-001i & 3.3858524697168449D-002 + 2.3714958079918535D-003i\\
 +2.9882165376338349D-001 -8.0995063945421741D-001i & 3.4768791677853524D-002 + 3.9881352204446229D-003i\\
 +5.9899770775166200D-002 -1.1948744535190035D+000i & 3.6149569249415578D-002 + 5.6587442927134319D-003i\\
 -2.7850156846250274D-001 -1.6446959027588257D+000i & 3.8019542860491549D-002 + 7.4059306342027370D-003i\\
 -7.3476207454759612D-001 -2.1838462972755117D+000i & 4.0404018032102237D-002 + 9.2533381404359474D-003i\\
 -1.3336627843342230D+000 -2.8416087014987785D+000i & 4.3335262806011098D-002 + 1.1225966955375540D-002i\\
 -2.1077320020767232D+000 -3.6537083955326728D+000i & 4.6852944413778977D-002 + 1.3350511787250059D-002i\\
 -3.0990120206246341D+000 -4.6642532348208814D+000i & 5.1004666075762921D-002 + 1.5655723156165313D-002i\\
 -4.3613425792163030D+000 -5.9281292955658031D+000i & 5.5846611194320855D-002 + 1.8172796462103385D-002i\\
 -5.9632850801149733D+000 -7.5139819214310553D+000i & 6.1444303658802629D-002 + 2.0935794138393353D-002i\\
 -7.9918463880908517D+000 -9.5079440821622647D+000i & 6.7873494551196165D-002 + 2.3982106603468702D-002i\\
 -1.0557204464776218D+001 -1.2018314543775274D+001i & 7.5221187251822183D-002 + 2.7352958248766927D-002i\\
 -1.3798692502929331D+001 -1.5181439937409575D+001i & 8.3586814817376690D-002 + 3.1093965310080390D-002i\\
 -1.7892366578989638D+001 -1.9169120201715121D+001i & 9.3083585564254306D-002 + 3.5255753171782450D-002i\\
 -2.3060567848491313D+001 -2.4197939613161662D+001i & 1.0384001506633338D-001 + 3.9894641457634002D-002i\\
 -2.9583998639207351D+001 -3.0541030203791603D+001i & 1.1600166529906714D-001 + 4.5073406179197625D-002i\\
 -3.7816968335105827D+001 -3.8542906477049684D+001i & 1.2973311446494412D-001 + 5.0862129255676444D-002i\\
 -4.8206637106159420D+001 -4.8638177144812161D+001i & 1.4522018415709023D-001 + 5.7339146901350516D-002i\\
 -6.1317302675453163D+001 -6.1375150182749685D+001i & 1.6267245400022495D-001 + 6.4592109714717016D-002i\\
 -7.7861049219288773D+001 -7.7445613273673700D+001i & 1.8232609779849215D-001 + 7.2719168815048321D-002i\\
 -9.8736423044743205D+001 -9.7722407114526959D+001i & 2.0444707957048894D-001 + 8.1830304077826238D-002i\\
 -1.2507723565031891D+002 -1.2330683231921709D+002i & 2.2933475272201531D-001 + 9.2048812443472808D-002i\\
 -1.5831414482743193D+002 -1.5558846474502963D+002i & 2.5732591106254793D-001 + 1.0351297643999625D-001i\\
 -2.0025235847869459D+002 -1.9632062801341692D+002i & 2.8879934648604044D-001 + 1.1637793549892739D-001i\\
 -2.5316968150772374D+002 -2.4771562239065099D+002i & 3.2418097499311860D-001 + 1.3081778538823258D-001i\\
 -3.1994023103558078D+002 -3.1256488222792592D+002i & 3.6394960042284941D-001 + 1.4702793417289575D-001i\\
 -4.0419053933177224D+002 -3.9439058811408586D+002i & 4.0864339389209720D-001 + 1.6522774658533385D-001i\\
 -5.1049652293304160D+002 -4.9763696830499521D+002i & 4.5886717662584403D-001 + 1.8566351259073638D-001i\\
 -6.4463201600277023D+002 -6.2791167965116256D+002i & 5.1530060473377282D-001 + 2.0861178031960287D-001i\\
 -8.1388236661047688D+002 -7.9229037823057672D+002i & 5.7870736669405309D-001 + 2.3438309847056926D-001i\\
 -1.0274401283961181D+003 -9.9970102193834418D+002i & 6.4994551800979161D-001 + 2.6332621882780033D-001i\\
 -1.2969043389592882D+003 -1.2614087777944171D+003i & 7.2997909289244700D-001 + 2.9583281576374076D-001i\\
 -1.6369105024555524D+003 -1.5916278709100643D+003i & 8.1989115010796476D-001 + 3.3234278659416305D-001i\\
 -2.0659254928655710D+003 -2.0082936066805410D+003i & 9.2089842952925771D-001 + 3.7335020451285178D-001i\\
 -2.6072505517969930D+003 -2.5340364973160094D+003i & 1.0343678177356390D+000 + 4.1941000466434852D-001i\\
 -3.2902868569904354D+003 -3.1974114007865637D+003i & 1.1618348454808407D+000 + 4.7114549383439813D-001i\\
 -4.1521323987955984D+003 -4.0344484308358847D+003i & 1.3050244673476217D+000 + 5.2925678538340226D-001i\\
 -5.2395968963294799D+003 -5.0906098731829061D+003i & 1.4658744047907917D+000 + 5.9453027356908161D-001i\\
 -6.6117441901645352D+003 -6.4232593945873459D+003i & 1.6465613684596623D+000 + 6.6784927547013706D-001i\\
 -8.3431001974123810D+003 -8.1047776557969619D+003i & 1.8495305146549577D+000 + 7.5020598452333265D-001i\\
 -1.0527700663992231D+004 -1.0226493547754852D+004i & 2.0775285345412993D+000 + 8.4271489743589800D-001i\\
 -1.3284198561072539D+004 -1.2903644569677639D+004i & 2.3336408238968738D+000 + 9.4662789617368170D-001i\\
 -1.6762308525499571D+004 -1.6281635763758291D+004i & 2.6213332364019597D+000 + 1.0633511891228977D+000i\\
 -2.1150938363365654D+004 -2.0543937150208803D+004i & 2.9444989854949983D+000 + 1.1944643406826623D+000i\\
 -2.6688449265921696D+004 -2.5922048598801452D+004i & 3.3075113294981566D+000 + 1.3417416468073389D+000i\\
 -3.3675602004538763D+004 -3.2708073362471001D+004i & 3.7152827529767691D+000 + 1.5071761457643933D+000i\\
\hline
\end{tabular}
}
\end{center}
\label{tab5}
\end{table}

\section{Conclusions}

We have developed an efficient separated sum-of-exponentials approximation
for the heat kernel in any dimension. The number of exponentials needed
is $\mathcal{O}(\log^2 (\frac{T}{\delta}))$ to obtain an approximation 
that is valid for $t\in [\delta, T]$ and $x\in \bbR^d$ 
for any prescribed precision $\epsilon$, but
only $\mathcal{O}(\log (\frac{T}{\delta}))$ of these 
modes involve the spatial variables. Such approximations can be  
combined with the local quadratures
of \cite{li2} and the fast algorithms of \cite{greengard1,greengard2,ho}
to create efficient, accurate and robust methods for the solution of 
boundary value problems governed by the heat equation in complex
geometry.

We have assumed here that the time step $\Delta t$ is fixed throughout
the simulation, but note that in many applications, adaptive time-stepping 
will be needed. 
We are currently developing heat solvers 
in both two and three dimensions that make use of the sum-of-exponentials
representation, permit adaptive time-steps, and incorporate fast algorithms for the 
spatial integrals.
We will report on their performance at a later date.


\begin{thebibliography}{10}

\bibitem{alpert0}
B. K. Alpert, 
\emph{Hybrid Gauss-trapezoidal quadrature rules}. 
SIAM J. Sci. Comput., {\bf 20} (1999), 1551--1584.

\bibitem{alpert1}
B. Alpert, L. Greengard, and T. Hagstrom, 
\emph{Rapid evaluation of nonreflecting boundary kernels for time-domain wave propagation}. 
SIAM J. Numer. Anal. {\bf 37} (2000), no. 4, 1138--1164.
 
\bibitem{alpert2}
B. Alpert, L. Greengard, and T. Hagstrom, 
\emph{Nonreflecting boundary conditions for the time-dependent wave equation}. 
J. Comput. Phys. {\bf 180} (2002), no. 1, 270--296.

\bibitem{beylkin1}
G.~Beylkin and L.~Monz\'{o}n, 
\emph{On Generalized Gaussian Quadrature for Exponentials and their Applications}.
Appl. Comput. Harmon. Anal. {\bf 12} (2002), 332-373.

\bibitem{beylkin2}
G.~Beylkin and L.~Monz\'{o}n, 
\emph{On approximation of functions by exponential sums}.
Appl. Comput. Harmon. Anal. {\bf 19} (2005), 17-48.

\bibitem{beylkin3}
G.~Beylkin and L.~Monz\'{o}n, 
\emph{Approximation by exponential sums revisited}.
Appl. Comput. Harmon. Anal. {\bf 28} (2010), 131-140.

\bibitem{braess1}
D. Braess and W. Hackbusch, 
\emph{Approximation of $1/x$ by exponential sums in $[1,\infty)$}. 
IMA J. Numer. Anal. {\bf 25} (2005), no. 4, 685--697.

\bibitem{braess2}
D. Braess and W. Hackbusch,
\emph{On the efficient computation of high-dimensional integrals and the approximation by exponential sums}. 
Multiscale, nonlinear and adaptive approximation, 39--74, Springer, Berlin, 2009. 

\bibitem{BRAT:MEIRON} 
K. Brattkus and D. I. Meiron,
{\em Numerical simulations of unsteady crystal growth},
SIAM J. Appl. Math. {\bf 52}, (1992), 1303--1320. 

\bibitem{BR:EDITOR} 
Brebbia, C. A. (ed.),
{\em Topics in Boundary Element Research}, Vol. 1,
 Springer-Verlag, Berlin (1981).

\bibitem{bremer}
J. Bremer, Z. Gimbutas, and V. Rokhlin, 
\emph{A nonlinear optimization procedure for generalized Gaussian quadratures}. 
SIAM J. Sci. Comput. {\bf 32} (2010), no. 4, 1761--1788. 

\bibitem{HSS}
S. Chandrasekaran, P. Dewilde, M. Gu, W. Lyons, and T. Pals,
{\em A fast solver for HSS representations via sparse matrices}, 
SIAM J. Matrix Anal. Appl. {\bf 29} (2006), 67--81.

\bibitem{chapko}
R. Chapko and R. Kress, 
\emph{Rothe's method for the heat equation and boundary integral equations}. 
J. Integral Equations Appl. {\bf 9} (1997), no. 1, 47--69. 

\bibitem{cheng}
H. Cheng, L. Greengard, and V. Rokhlin, 
\emph{A fast adaptive multipole algorithm in three dimensions}. 
J. Comput. Phys. {\bf 155} (1999), no. 2, 468--498. 

\bibitem{costabel}
M. Costabel,
{\em Time-dependent problems with the boundary integral equation 
method}. Chapter 25, 
Encyclopedia of Computational Mechanics,
(E. Stein, R. De Borst, T.J.R. Hughes, eds), John Wiley, 2004.

\bibitem{FRIEDMAN} 
A. Friedman,
{\em Partial Differential Equations of Parabolic Type},
Prentice-Hall, Englewood Cliffs, New Jersey, 1964.

\bibitem{gillman}
A. Gillman, P. Young, and P. G. Martinsson, 
{\em A direct solver with O(N) complexity for integral equations 
on one-dimensional domains}, 
Front. Math. China, {\bf 7} (2012), 217--247.

\bibitem{greengard1}
L. Greengard and P. Lin,  
{\em Spectral approximation of the free-space heat kernel}.
Appl. Comput. Harmon. Anal.,  {\bf 9} (2000), 83--97. 

\bibitem{greengard2}
L. Greengard and J. Strain,  
{\em A fast algorithm for the evaluation of heat potentials}.
Comm. Pure Appl. Math., {\bf 43} (1990), 949--963.

\bibitem{GL:PDE} 
R. B. Guenther and J. W. Lee,
{\em Partial Differential Equations of Mathematical Physics and 
     Integral Equations}.
Prentice-Hall, Englewood Cliffs, New Jersey, 1988.

\bibitem{Hmat}
W. Hackbusch and S. B\"{o}rm,
{\em Data-sparse approximation by adaptive H2-matrices}, 
Computing {\bf 69} (2002), no. 1, 1--35.

\bibitem{hagstrom}
T. Hagstrom,
{\em Radiation boundary conditions for the numerical simulation of waves}, 
Acta Numerica, {\bf 8} (1999), 47--106.

\bibitem{ho}
K. L. Ho and L. Greengard,
\emph{A fast direct solver for structured linear systems by recursive skeletonization}.
 SIAM J. Sci. Comput. {\bf 34} (2012), no. 5, A2507--A2532.

\bibitem{ibanez}
M. T. Ibanez and H. Power,
\emph{An efficient direct BEM numerical scheme for phase change problems 
using Fourier series}.
Comput. Methods Appl. Mech. Engrg. {\bf 191} (2002), 2371--2402.

\bibitem{jiang1}
S. Jiang and L. Greengard,
\emph{Fast evaluation of nonreflecting boundary conditions for the Schr\"{o}dinger 
equation in one dimension}. 
Comput. Math. Appl. {\bf 47} (2004), no. 6-7, 955--966.

\bibitem{jiang2}
S. Jiang and L. Greengard, 
\emph{Efficient representation of nonreflecting boundary conditions for the time-dependent 
Schr\"{o}dinger equation in two dimensions}. 
Comm. Pure Appl. Math. {\bf 61} (2008), no. 2, 261--288.

\bibitem{jiang3}
S. Jiang, S. Veerapaneni, and L. Greengard,
\emph{Integral equation methods for unsteady Stokes flow in two dimensions}.
SIAM J. Sci. Comput. {\bf 34} (2012), no. 4, A2197-A2219.

\bibitem{qbx}
A. Kloeckner, A. Barnett, L. Greengard, and M. ONeil, 
{\em Quadrature by Expansion: A New Method for the Evaluation 
of Layer Potentials}.
arXiv:1207.4461, 2012.

\bibitem{kress}
R. Kress,
\emph{Linear integral equations}.
Second Edition, Springer, 1999.

\bibitem{li1}
J. R. Li, and L. Greengard, 
\emph{On the numerical solution of the heat equation. I. Fast solvers in free space}. 
J. Comput. Phys. {\bf 226} (2007), no. 2, 1891--1901.

\bibitem{li2}
J. Li and L. Greengard, 
{\em High order accurate methods for the evaluation of layer heat potentials}.
SIAM J. Sci. Comput., {\bf 31} (2009), 3847--3860.

\bibitem{lopez1}
M. L\'{o}pez-Fern\'{a}ndez and C. Palencia, 
\emph{On the numerical inversion of the Laplace transform of certain holomorphic mappings}.
Applied Numer. Math. {\bf 51} (2004), 289-303.

\bibitem{lopez2}
M. L\'{o}pez-Fern\'{a}ndez, C. Palencia, and A. Sch\"{a}dle, 
\emph{A spectral order method for inverting sectorial Laplace transforms}.
SIAM J. Numer. Anal. {\bf 44} (2006), no. 3, 1332-1350.

\bibitem{lopez3}
M. L\'{o}pez-Fern\'{a}ndez, C. Lubich, and A. Sch\"{a}dle, 
\emph{Adaptive, fast, and oblivious convolution in evolution equations with memory}.
SIAM J. Sci. Comput. {\bf 30} (2008), no. 2, 1015-1037.

\bibitem{lubich}
C. Lubich and A. Sch\"{a}dle,
{\em Fast convolution for nonreflecting boundary conditions}.
SIAM J. Sci. Comput. {\bf 24} (2002), 161--182.

\bibitem{MP:EDITORS} 
Morino, L. and Piva, R. (eds.),
{\em Boundary Integral Methods: Theory and Applications}.
Springer-Verlag, Berlin, 1990.

\bibitem{ma}
J. Ma, V. Rokhlin, and S. Wandzura, 
\emph{Generalized Gaussian quadrature rules for systems of arbitrary functions}. 
SIAM J. Numer. Anal. {\bf 33} (1996), no. 3, 971--996.

\bibitem{mcintyre}
E. A. McIntyre,
\emph{Boundary integral solutions of the heat equation}. 
Math. Comp. {\bf 46} (1986), no. 173, 71--79, S1--S14.
 
\bibitem{Noon} 
P. J. Noon,
{\em The Single Layer Heat Potential and Galerkin Boundary Element 
     Methods for the Heat Equation}. 
Ph.D. Thesis, University of Maryland, 1988.

\bibitem{POGOR} 
W. Pogorzelski,
{\em Integral Equations and Their Applications}.
Pergamon Press, Oxford, 1966.


\bibitem{schadle}
A. Sch\"{a}dle, M. L\'{o}pez-Fern\'{a}ndez, and C. Lubich, 
\emph{Fast and oblivious convolution quadrature}. 
SIAM J. Sci. Comput. {\bf 28} (2006), no. 2, 421-438.

\bibitem{SETHSTR} 
J. A. Sethian and Strain, J.,
{\em Crystal Growth and Dendritic Solidification}.
J. Comput. Phys. {\bf 98}, (1992), 231--253.

\bibitem{stenger1}
F. Stenger,
\emph{Approximation via Whitaker's Cardinal Functions}.
J. Approx. Theory {\bf 17} (1976), 222-240.

\bibitem{stenger2}
F. Stenger,
\emph{Numerical methods based on Whitaker Cardinal, or sinc Functions}.
SIAM Rev. {\bf 23} (1981), 165-224.

\bibitem{tausch}
J. Tausch,  
\emph{A fast method for solving the heat equation by layer potentials}. 
J. Comput. Phys. {\bf 224} (2007), no. 2, 956--969.

\bibitem{talbot}
A. Talbot,
\emph{The accurate numerical inversion of Laplace transforms}.
J. Inst. Math. Appl. {\bf 23} (1979), 97-120.

\bibitem{trefethen}
L. N. Trefethen, J. A. C. Weideman, and T. Schmelzer, 
\emph{Talbot quadratures and rational approximations}.
BIT Numer. Math. {\bf 46} (2006), 653-670.

\bibitem{xu}
K. Xu and S. Jiang, 
\emph{A Bootstrap Method for Sum-of-Poles Approximations}. 
J. Sci. Comput. {\bf 55} (2013), 16--39.

\bibitem{yarvin}
N. Yarvin and V. Rokhlin, 
\emph{An improved fast multipole algorithm for potential fields on the line}. 
SIAM J. Numer. Anal. {\bf 36} (1999), no. 2, 629--666.

\bibitem{kifmm}
L. Ying, G. Biros, and D. Zorin, 
{\em A kernel-independent adaptive fast multipole algorithm in two 
and three dimensions}. 
J. Comput. Phys. {\bf 196} (2004), 591--626.

\end{thebibliography}
\end{document}